\documentclass[a4,12pt]{amsart}
\usepackage{amssymb,amstext,amsmath,amscd,amsthm,amsfonts,enumerate,graphicx,latexsym}
\usepackage[usenames]{color}
\usepackage[all]{xy}

\newtheorem{thm}{Theorem}[section]
\newtheorem{cor}[thm]{Corollary}
\newtheorem{lem}[thm]{Lemma}
\newtheorem{prop}[thm]{Proposition}
\theoremstyle{definition}

\newtheorem{rem}[thm]{Remark}

\newtheorem{ex}[thm]{Example}

\newtheorem{nota}[thm]{Notation}

\newtheorem*{claim*}{Claim}
\theoremstyle{remark}
\newtheorem*{ac}{Acknowlegments}
\numberwithin{equation}{thm}
\def\height{\operatorname{\mathsf{ht}}}

\def\Ext{\operatorname{\mathsf{Ext}}}

\def\syz{\mathsf{\Omega}}

\def\tr{\mathsf{Tr}}
\def\depth{\operatorname{\mathsf{depth}}}
\def\dim{\operatorname{\mathsf{dim}}}

\def\Hom{\operatorname{\mathsf{Hom}}}
\def\depth{\operatorname{\mathsf{depth}}}

\def\End{\operatorname{\mathsf{End}}}

\def\ocm{\mathrm{\Omega CM}}
\def\cm{\mathrm{CM}}
\def\Ref{\mathrm{Ref}}
\def\mod{\operatorname{mod}}
\def\add{\operatorname{add}}
\begin{document}
\allowdisplaybreaks
\setlength{\baselineskip}{14.95pt}
\title[Syzygies of CM modules over one dimensional CM local rings]{Syzygies of Cohen--Macaulay modules over one dimensional Cohen--Macaulay local rings}
\author{Toshinori Kobayashi}
\address{Graduate School of Mathematics, Nagoya University, Furocho, Chikusaku, Nagoya, Aichi 464-8602, Japan}
\email{m16021z@math.nagoya-u.ac.jp}
\thanks{2010 {\em Mathematics Subject Classification.} 13C14, 13D02, 13H10}
\thanks{{\em Key words and phrases.} Cohen--Macaulay ring/module, syzygy, almost Gorenstein ring}
\begin{abstract}
We study syzygies of (maximal) Cohen--Macaulay modules over one dimensional Cohen--Macaulay local rings. We compare these modules to Cohen--Macaulay modules over the endomorphism ring of the maximal ideal. After this comparison, we give several characterizations of  almost Gorenstein rings in terms of syzygies of Cohen--Macaulay modules.
\end{abstract}
\maketitle
%\tableofcontents
%%%%%%%%%%%%%%%%%%%%%%%%%%%%%%%%%%%%%%%%%%%%%%%%%%%%%%%%
\section{Introduction}

Throughout this paper, let $(R,\mathfrak{m})$ be a $1$-dimensional singular local Cohen--Macaulay ring which is generically Gorenstein, and set $E=\End_R(\mathfrak{m})$. All subcategories are assumed to be full. First of all, we give some notations.

\begin{nota}
Let $A$ be a commutative Noetherian ring. We denote:
\begin{enumerate}[(1)]\item
by $\mod (A)$ the category of finitely generated $A$-modules,
\item
by $\cm(A)$ the subcategory of Cohen--Macaulay $A$-modules,
\item
by $\cm'(A)$ the subcategory of Cohen--Macaulay $A$-modules without free summands,
\item
by $\ocm(A)$ the subcategory of first syzygies of Cohen--Macaulay $A$-modules, and
\item
by $\ocm'(A)$ the subcategory $\ocm(A)\cap \cm'(A)$.
\end{enumerate}
\end{nota}

Here we say that a finitely generated $A$-module $M$ is {\it Cohen--Macaulay} if the localization $M_\mathfrak{p}$ satisfies $\depth M_\mathfrak{p}\geq\height A_\mathfrak{p}$ for all prime ideals $\mathfrak{p}$ of $A$.

The category $\ocm(A)$ has been used to study surface singularities. For example, the relationship between $\ocm(A)$ and the category of special Cohen--Macaulay modules is established in \cite{IW}. Furthermore, the following theorem is shown in \cite[Corollary 3.3]{ncr}.

\begin{thm} [Dao--Iyama--Takahashi--Vial]
Let $A$ be an excellent henselian local normal domain of dimension two with algebraically closed residue field. Then $A$ has a rational singularity if and only if $\ocm(A)$ is of finite type.
\end{thm}

Here we say that a subcategory $\mathcal{X}$ of $\mod(A)$ is {\it of finite type} if there is a module $M\in\mod(A)$ with $\add(M)=\mathcal{X}$, where $\add (M)$ is the additive closure of $M$ in $\mod(A)$.

It is natural to ask what kind of one-dimensional local ring $R$ is such that $\ocm(R)$ is of finite type. To answer this question, we investigate the structure of $\ocm(R)$ to get the following result.

\begin{thm} \label{1}
The endomorphism ring $E$ is a commutative Cohen--Macaulay semilocal ring, and the following hold true.
\begin{enumerate}[\rm(1)]
\item
One has $\ocm(E)\subseteq\ocm'(R)\subseteq \cm(E)\subseteq\cm'(R)$ as subcategories of $\mod(R)$.
\item Assume that the residue field $R/\mathfrak{m}$ is infinite. Then
$\ocm(E)=\ocm'(R)$ if and only if $R$ has minimal multiplicity.
\item Assume that the completion $\widehat{R}$ of $R$ is generically Gorenstein. Then
$\cm(E)=\ocm'(R)$ if and only if $R$ is almost Gorenstein in the sense of \cite{GMP}.
\end{enumerate}
\end{thm}

The purpose of this paper is to give a proof of this theorem. As an application, we construct some examples of rings $R$ such that $\ocm(R)$ is of finite type. We also deduce the following results.

\begin{cor}
Assume that $R$ is almost Gorenstein. Then $\ocm(R)$ is of finite type if and only if $\cm(E)$ is of finite type.
\end{cor}

We remark that there is an example of a ring $R$ which is not almost Gorenstein such that $\ocm(R)$ is of finite type (see Example \ref{ex2}). 
As another application, we also obtain a result on the dimension of a triangulated category, which has been introduced by Rouquier \cite{R}. It turns out that each integer is realized as the dimension of the stable category of Cohen--Macaulay modules.

\begin{cor}
Let $n>0$ be any integer. Let $R=k[[H]]$ be the numerical semigroup ring with $H=\langle 2^{n+1}, \{2^{n+1}+2^i\}_{i=1,\dots,n},2^{n+1}+3\rangle$. Then $R$ is a complete intersection with
\[
\dim \underline{\cm}(R)=n,
\]
where $\underline{\cm}(R)$ stands for the stable category of $\cm(R)$.
\end{cor}

\section{The endomorphism ring of the maximal ideal}

In this section, we give a proof of Theorem \ref{1}. 
Let $A$ be a commutative Noetherian ring. Recall the definition of reflexive modules. A finitely generated $A$-module $M$ is called {\it reflexive} if the natural map $M\to \Hom_A(\Hom_A(M,A),A)$ is an isomorphism. We denote by $\Ref(A)$ the subcategory of reflexive $A$-modules. 
For any $M\in\mod (A)$, we denote by $\tr M$ the Auslander transpose of $M$.
For an integer $n\geq 1$, we define $\mathcal{F}_n(A)=\{M\mid \Ext_R^i(\tr M,R)=0 \text{ for }i=1,\dots,n\}$ as a full subcategory of $\mod(A)$. A module $M\in \mod(A)$ is called {\it $n$-torsionfree} if it is in $\mathcal{F}_n$.

We have the following characterization of modules in $\cm(A)$, $\ocm(A)$, or $\Ref(A)$. 

\begin{lem} \label{01}
\begin{enumerate}[\rm(1)]
\item
One has $\Ref(A)=\mathcal{F}_2(A)$.
\item
Assume that $A$ is a Cohen--Macaulay ring of dimension $d$ such that $A_\mathfrak{p}$ is Gorenstein for every prime ideal $\mathfrak{p}$ of $A$ with $\height \mathfrak{p}$ being at most $d-1$. Then $\mathcal{F}_d(A)=\cm(A)$ and $\mathcal{F}_{d+1}(A)=\ocm(A)$.
\end{enumerate}
\end{lem}

\begin{proof}
(1)
See \cite[Proposition 12.5]{LW} for instance.

(2)
See \cite[Theorems 3.6 and 3.8]{EG} to prove the equality $\mathcal{F}_d(A)=\cm(A)$ holds. The proof of \cite[Proposition 2.4]{IW} shows that $\mathcal{F}_{d+1}(A)=\ocm(A)$ holds.
\end{proof}

 We denote by $Q(R)$ the total quotient ring of $R$ and by $\overline{R}$ the integral closure of $R$ in $Q(R)$. Then we can identify the ring $E(=\End_R(\mathfrak{m}))$ with $\mathfrak{m}:_{Q(R)}\mathfrak{m}$ in the natural way. Note that $E$ is a one-dimensional Cohen--Macaulay semilocal ring contained in $\overline{R}$ because $E$ is a Cohen--Macaulay $R$-module. Using the propositions in previous section, we have the following identifications.

\begin{lem} \label{03}
$\ocm(R)=\mathcal{F}_2(R)=\Ref(R)$ and $\ocm(E)=\mathcal{F}_2(E)=\Ref(E)$.
\end{lem}

\begin{proof}
Since $R$ is $1$-dimensional and generically Gorenstein, the assumption of Lemma \ref{01} (2) is satisfied. Thus $\ocm(R)=\mathcal{F}_2(R)$. The ring $E$ satisfies the same condition, because the total quotient ring of $E$ coincides with $Q(R)$. Therefore $\ocm(E)=\mathcal{F}_2(E)$. Finally, we have $\mathcal{F}_2(R)=\Ref(R)$ and $\mathcal{F}_2(E)=\Ref(E)$ from Lemma \ref{01} (1).
\end{proof}

To prove Theorem \ref{1}, we will need the following lemma which is observed in Bass's ``ubiquity'' paper \cite{B}.
For the proof, see \cite[Lemma 4.9]{LW}.

\begin{lem} \label{1b}
Let $M\in \ocm'(R)$. Then $M$ has an $E$-module structure which is compatible with the action of $R$ on $M$.
\end{lem}

The module-finite ring extension $R\subseteq E$ induces an inclusion $\cm(E)\subseteq \cm(R)$ since $E$ is Cohen-Macaulay over $R$. Via this inclusion and the following lemma, we can view $\cm(E)$ as a subcategory of $\cm(R)$.

\begin{lem}
Let $N,M\in \cm(E)$. Then $\Hom_R(M,N)=\Hom_E(M,N)$.
\end{lem}

We prepare the following three lemmas about $\ocm(R)$. For a finitely generated $R$-module $M$, we denote by $\mu(M)$ the minimal number of generators of $M$ and by $\syz M$ the (first) syzygy of $M$ in the minimal free resolution.
The first lemma follows from \cite[Lemma 2.1]{K}.

\begin{lem} \label{3}
Let $M$ be a Cohen--Macaulay $R$-module. Then $\syz M$ has no free summand.
\end{lem}

\begin{lem} \label{4}
Let $0\to L \to M \xrightarrow[]{\alpha} N \to 0$ be an exact sequence of $R$-modules in $\cm(R)$. If $M$ is in $\ocm'(R)$, then so is $L$.
\end{lem}

\begin{proof}
Since $M$ is in $\ocm'(R)$, there is an exact sequence $0 \to M \xrightarrow[]{\beta} R^{\oplus a} \xrightarrow[]{\gamma} M' \to 0$ with a Cohen--Macaulay $R$-module $M'$. As $M$ has no free summand, $\gamma$ is a minimal free cover. In particular, $\mu(M')=a$. Taking the pushout of homomorphisms $\beta$ and $\gamma$, we have the following diagram.
$$
\xymatrix{
& & 0 \ar[d] & 0 \ar[d] &\\
0 \ar[r] & L \ar@{=}[d] \ar[r] & M \ar[d]^\beta \ar[r]^\alpha & N \ar[d] \ar[r] & 0\\
0 \ar[r] & L \ar[r] & R^{\oplus a} \ar[d]^\gamma \ar[r]^\delta & P \ar[d] \ar[r] & 0\\
& & M' \ar[d] \ar@{=}[r] & M' \ar[d] &\\
& & 0 & 0 &
}
$$
The surjections $R^{\oplus a}\to P$ and $P\to M'$ show that $a=\mu(M')\leq\mu(P)\leq a$. Thus $\delta$ is a minimal free cover and $L=\syz P$. Lemma \ref{3} implies that $L$ has no free summand and hence $L$ is in $\ocm'(R)$.
\end{proof}

\begin{lem} \label{1a}
The $R$-modules $\mathfrak{m}$ and $E$ belong to $\ocm'(R)$.
\end{lem}

\begin{proof}
Take a non zerodivisor $x$ of $R$. Then $R/(x)$ has depth zero. Therefore there is a short exact sequence
$
0 \to k \to R/(x) \to N \to 0
$
with some $R$-module $N$. Applying horseshoe lemma to this sequence, we have an exact sequence
$
0 \to \syz k \to R^{\oplus n} \to \syz N \to 0
$
for some $n\geq 1$. Since $\syz N$ is Cohen--Macaulay and $\syz k=\mathfrak{m}$, it follows that $\mathfrak{m}\in \ocm(R)$. This implies that $\mathfrak{m}$ is reflexive over $R$ by Lemma \ref{03}. By \cite[ Corollary 5.7]{ADS}, $\mathfrak{m}$ has no free summand. It yields that $E=\Hom_R(\mathfrak{m},\mathfrak{m})=\Hom_R(\mathfrak{m},R)$. Thus $E=\syz^2 \tr \mathfrak{m}=\syz(\syz \tr\mathfrak{m})$. In particular, $E$ is in $\ocm(R)$. If $E$ has a nonzero free summand, then, since $\mathfrak{m}$ is reflexive and $E=\Hom_R(\mathfrak{m},R)$, $\mathfrak{m}$ must have a nonzero free summand. This is a contradiction. As a consequence, $\mathfrak{m}$ and $E\in\ocm'(R)$.
\end{proof}

Now we can give a proof of Theorem \ref{1} (1).

\begin{proof}[Proof of Theorem \ref{1} {\rm(1)}]
To show the inclusion $\cm(E)\subseteq \cm'(R)$, we only need to prove that all modules in $\cm(E)$ have no $R$-free summand. Assume that $M\in \cm(E)$ has a nonzero $R$-free summand. Then $E$ also has a nonzero $R$-free summand, since there is a surjection $E^{\oplus n} \to M$. This contradicts Lemma \ref{1a}.

The inclusion $\ocm'(R)\subseteq \cm(E)$ follows by Lemma \ref{1b}.

Finally, we show $\ocm(E)\subseteq \ocm'(R)$. Let $M$ be an $E$-module in $\ocm(E)$. Then we have an exact sequence $0 \to M \to E^{\oplus n} \to N \to 0$ with some $E$-module $N$ in $\cm(E)$. In view of Lemmas \ref{3} and \ref{4}, this sequence yields that $M$ is in $\ocm'(R)$.
\end{proof}

The $R$-module $\mathfrak{m}$ plays an important role to prove the theorem. Next lemma says that $\mathfrak{m}$ can be regarded as a ``cogenerator'' of $\ocm'(R)$.

\begin{lem} \label{7}
Let $M$ be an $R$-module in $\ocm'(R)$. Then there is an exact sequence
\[
0 \to M \to \mathfrak{m}^{\oplus n} \to N \to 0
\]
of modules in $\cm(E)$.
\end{lem}

\begin{proof}
As $M$ is in $\ocm'(R)$, we have an exact sequence $0 \to M \xrightarrow[]{\alpha} R^{\oplus n} \to N' \to 0$ with a Cohen-Macaulay $R$-module $N$. Then, since $M$ has no free summand, there is a homomorphism $\beta:M \to \mathfrak{m}$ such that $\alpha=i\circ \beta$, where $i$ is the natural inculusion $\mathfrak{m}^{\oplus n} \to R^{\oplus n}$. Let $N$ be the cokernel of $\beta$. We have the following commutative diagram with exact rows and columns.
\[
\xymatrix{
0 \ar[r] & M \ar[r]^\alpha & R^{\oplus n} \ar[r] & N' \ar[r] & 0\\
0 \ar[r] & M \ar@{=}[u] \ar[r]^\beta & \mathfrak{m}^{\oplus n} \ar[u]^i \ar[r] & N \ar[u] \ar[r] & 0\\
& & 0 \ar[u] & 0 \ar[u] &
}
\]
Since $\beta \in \Hom_R(M,\mathfrak{m})=\Hom_E(M,\mathfrak{m})$, $N$ is a $E$-module. The exactness of $0\to N \to N'$ implies that $N$ is Cohen-Macaulay over $R$. Thus $N$ is in $\cm(E)$.
\end{proof}

To prove Theorem \ref{1} {\rm (2)}, we will use the following lemma (cf. \cite[Corollary 3,2]{BV}).

\begin{lem} \label{2}
Let $A$ be a $1$-dimensional Cohen--Macaulay generically Gorenstein ring and $I$ be an ideal of $A$ containing a non zerodivisor such that $\Hom_A(I,I)$ is naturally isomorphic to $A$. If $I$ is reflexive, then it is isomorphic to $A$.
\end{lem}

\begin{proof}
Denote by $\sigma$ the natural homomorphism $I \to \Hom_A(\Hom_A(I,A),A)$, which maps $a$ to $[\phi \mapsto \phi(a)]$ for $a\in I$ and $\phi\in \Hom_A(I,A)$. Then $\sigma$ is isomorphism by the reflexivity of $I$. Let $\mathrm{ev}:I\otimes_A \Hom_A(I,A) \to A$ be the evaluation homomorphism, which maps $a\otimes \phi$ to $\phi(a)$. Let $K$ be the kernel of $\mathrm{ev}$, $C$ be the cokernel of $\mathrm{ev}$, and $T$ be the image of $\mathrm{ev}$. we have the following exact sequences.
\begin{align}
0 \to K \to I\otimes_A \Hom_A(I,A) \to T \to 0\\
0 \to T \to A \to C \to 0
\end{align}
Let $Q(A)$ be the total quotient ring of $A$. Then the ideal $I\otimes_A Q(A)$ of $Q(A)$ is isomorphic to $Q(A)$. Thus $\mathrm{ev}\otimes Q(A)$ is an isomorphism. Then $K\otimes_A Q(A)=0=C\otimes_A Q(A)$, and it follows from Cohen-Macaulayness of $A$ that $\Hom_A(K,A)=0=\Hom_A(C,A)$. Applying $\Hom_A(-,A)$ to the above sequences, we have $\Hom_A(T,A)\cong \Hom_A(I\otimes_A \Hom_A(I,A),A)$ and the following exact sequence
\[
0 \to A=\Hom_A(A,A) \xrightarrow[]{f} \Hom_A(T,A) \to \Ext^1_A(C,A)\to 0.
\]
Via the isomorphism $\Hom_A(T,A)\cong \Hom_A(I\otimes_A \Hom_A(I,A),A)$, the map $f$ corresponds to the map $\Hom_A(\mathrm{ev},A):A \to \Hom_A(I\otimes_A \Hom_A(I,A),A)$, which maps $1_A$ to $\mathrm{ev}$. By the hom-tensor adjointness and composition with $\Hom_A(I,\sigma^{-1})$, we get an isomorphism
\[
\Psi:\Hom_A(I\otimes_A \Hom_A(I,A),A)\rightarrow \Hom_A(I,\Hom_A(\Hom_A(I,A),A)\rightarrow \Hom_A(I,I).
\]
Consider the composition $\Psi\circ\Hom_A(\mathrm{ev},A):A \to \Hom_A(I,I)$. Then it maps $1_A$ to $\mathrm{id}_I$. Therefore $\Psi\circ\Hom_A(\mathrm{ev},A)$ is an isomorphism by the assumption. It yields that $\Hom_A(\mathrm{ev},A)$ and $f$ are isomorphisms and hence $\Ext^1_A(C,A)=0$. Since $A$ is $1$-dimensional, we obtain that $C=0$. In this case, $\mathrm{ev}$ is a surjection and we can see that there is a surjective map $I\to A$. This yields that $I$ is isomorphic to $A$.
\end{proof}

Now we show the Theorem \ref{1} {\rm(2)}.

\begin{proof}[Proof of Theorem \ref{1} {\rm(2)}]
Assume the equality $\ocm(E)=\ocm'(R)$ holds. Then $\mathfrak{m}\in \ocm'(R)=\Ref(E)$. Since $\Hom_E(\mathfrak{m},\mathfrak{m})=\End_R(\mathfrak{m})=E$, we can apply Lemma \ref{2} to see that $\mathfrak{m}$ is isomorphic to $E$. By \cite[Proposition 2.5]{O}, it follows that $R$ has minimal multiplicity.

Conversely, assume that $R$ has minimal multiplicity. Then $\mathfrak{m}$ is isomorphic to $E$ by using \cite[Proposition 2.5]{O} again. Let $M\in \ocm'(R)$. Lemma \ref{7} yields that there is an exact sequence $0 \to M \to \mathfrak{m}^{\oplus n} \to N \to 0$ with $N\in\cm(E)$. Since $\mathfrak{m}^{\oplus n}\cong E^{\oplus n}$, we have $M\in \ocm(E)$. Thus $\ocm(E)=\ocm'(R)$.
\end{proof}

We denote by $\omega$ a canonical module of $R$ and set $(-)^\dag=\Hom_R(-,\omega)$. If $\omega$ exists, then we can give an equivalent condition to the equality $\cm(S)=\ocm'(R)$ by using the canonical dual $(-)^\dag$.

\begin{lem} \label{5}
Assume that $R$ has a canonical module $\omega$. Then the equality $\cm(E)=\ocm'(R)$ holds if and only if $E^\dag \in \ocm'(R)$.
\end{lem}

\begin{proof}
The ``only if" part is clear. Now we assume $E^\dag \in \ocm'(R)$. Let $M$ be in $\cm(E)$. Taking a free cover of $M^\dag$ over $E$, we get an exact sequence $0 \to N\to E^{\oplus n} \to M^\dag \to 0$ with some $E$-module $N$. Since $M^\dag,E\in\cm(R)$, $N$ is also in $\cm(R)$. Applying $(-)^\dag$ to this sequence, we have an exact sequence $0 \to M \to (E^\dag)^{\oplus n} \to N^\dag \to 0$. Using Lemma \ref{4}, $M$ is in $\ocm'(R)$. This shows that $\cm(E)=\ocm'(R)$.
\end{proof}

If the completion $\widehat{R}$ of $R$ is generically Gorenstein, then $R$ has a canonical module by \cite[Proposition 2.7]{GMP}. In this situation, we see in the next lemma that the condition $\cm(E)=\ocm'(R)$ is stable under flat local extension.

\begin{cor} \label{6}
Let $\varphi:(R,\mathfrak{m})\to(R',\mathfrak{m}')$ be a flat local homomorphism such that $\mathfrak{m}R'=\mathfrak{m}'$. Assume that the completion $\widehat{R}$ of $R$ is generically Gorenstein. Then $\cm(E)=\ocm'(R)$ if and only if $\cm(\End_{R'}(\mathfrak{m}'))=\ocm'(R')$.
\end{cor}

\begin{proof}
Let $E'=\End_{R'}(\mathfrak{m}')$. Note that $\widehat{R'}$ is also generically Gorenstein by \cite[Proposition 2.12]{GMP}. In addition, $\omega\otimes_R R'$ is a canonical module of $R'$. Therefore, by Lemma \ref{5}, $\cm(E)=\ocm'(R)$ if and only if $E^\dag \in \ocm'(R)$, and $\cm(E')=\ocm'(R')$ if and only if $\Hom_{R'}(E',\omega\otimes_R R') \in \ocm'(R)$. Here $E'=E\otimes_R R'$ and hence $\Hom_{R'}(E',\omega\otimes_R R')=(E^\dag)\otimes_R R'$. Lemma \ref{03} implies that $E^\dag \in \ocm'(R)$ if and only if $\Ext^1_R(\tr (E^\dag),R)=0=\Ext^2_R(\tr (E^\dag),R)$. The later condition is equivalent to the equations $\Ext^1_{R'}(\tr ((E^\dag)\otimes_R R'),R')=0=\Ext^2_{R'}(\tr (E^\dag)\otimes_R R'),R')$ since $R\to R'$ is faithfully flat and the Auslander transpose is preserved under a base change.
\end{proof}

Using the above lemma, we can replace $R$ with the completion $\widehat{R}$.
We have one more equivalent condition to being $\cm(E)=\ocm'(R)$.

\begin{lem} \label{8}
Assume that the completion $\widehat{R}$ of $R$ is generically Gorenstein. Then $E^\dag \in \ocm'(R)$ if and only if $E^\dag \cong \mathfrak{m}$.
\end{lem}

\begin{proof}
Thanks to Corollary \ref{6}, we can assume that $R$ is complete. If $E^\dag \cong \mathfrak{m}$, then we have $E^\dag \in \ocm'(R)$. Conversely, we assume $E^\dag\in\ocm'(R)$. Using Lemma \ref{7}, we get an exact sequence
\begin{equation} \label{9}
0 \to E^\dag \xrightarrow[]{\alpha} \mathfrak{m}^{\oplus m} \to N \to 0
\end{equation}
of modules in $\cm(E)$. By the Krull-Schmidt theorem on $R$, we have a unique decomposition $\mathfrak{m}=\mathfrak{m}_1\oplus\cdots\oplus\mathfrak{m}_n$, where $\mathfrak{m}_i$ are indecomposable $R$-modules. Then we obtain $E=\End_R(\mathfrak{m}_1)\times\cdots\times\End_R(\mathfrak{m}_n)$ as an $R$-algebra. The components $E_i=\End_R(\mathfrak{m}_i)$ of $E$ are local rings because of the indecomposablity of $\mathfrak{m}_i$. Set $\mathfrak{n}_i$ the maximal ideal of $E$ corresponding to the maximal ideal of $E_i$. Note that the localization $(E^\dag)_{\mathfrak{n}_i}=(\Hom_R(E,\omega))_{\mathfrak{n}_i}$ is the canonical module of $E_i$ and the localization $\mathfrak{m}_{\mathfrak{n}_i}$ is equal to $(\mathfrak{m}_i)_{\mathfrak{n}_i}$. Thus, after localizing at $\mathfrak{n}_i$, the sequence (\ref{9}) becomes split exact and $(E^\dag)_{\mathfrak{n}_i}$ is a direct summand of $(\mathfrak{m}_i)_{\mathfrak{n}_i}^{\oplus m}$. The modules $(E^\dag)_{\mathfrak{n}_i}$ and $(\mathfrak{m}_i)_{\mathfrak{n}_i}$ are both indecomposable. Hence we obtain an isomorphism $(E^\dag)_{\mathfrak{n}_i}\cong (\mathfrak{m}_i)_{\mathfrak{n}_i}$ by the Krull-Schmidt theorem. The homomorphism $\alpha$ is a split injection, since it becomes a split injection after localizing at $\mathfrak{n}_i$ for all $i=1,\dots,n$. Therefore $E^\dag$ is isomorphic a direct summand of $\mathfrak{m}^{\oplus m}$. Set $E^\dag\cong \mathfrak{m}_1^{\oplus a_1}\oplus\cdots\oplus \mathfrak{m}_n^{\oplus a_n}$. Then the localization at $\mathfrak{n}_i$ shows that $a_i=1$. Consequencely, $E^\dag\cong \mathfrak{m}_1\oplus\cdots\oplus\mathfrak{m}_n=\mathfrak{m}$.
\end{proof}

The following lemma will be used to prove Theorem \ref{1} {\rm(3)}.
\begin{lem} \label{11}
Let $A$ be a ring with total quotient ring $T$, $\overline{A}$ be the integral closure of $A$ in $T$, and $X$ be an $A$-submodule of $\overline{A}$ containing $A$. If there is an isomorphism $\phi:A \to X$ of $A$-modules, then $X=A$. 
\end{lem}

\begin{proof}
Let $i:A\to X$ be the inclution homomorphism. Then $\phi^{-1}\circ i:A \to A$ is an endmorphism of $A$. Hence it is a multiplication map by $r$ for some $r\in A$. Since $r=\phi^{-1}\circ i:A\to A$ is injective, $1/r$ is in $T$. We have $1=i(1)=\phi(r)=r\phi(1)$ in $\overline{A}$ and hence $1/r=\phi(1)\in \overline{A}$. It means that $1/r$ is integral over $A$. Therefore we have an equation of integral dependence
\[
(1/r)^n+a_1(1/r)^{n-1}+\cdots+a_n=0,
\]
where $a_i\in A$ for all $i=1,\dots,n$. Multiplying $r^n$, we get
$
1+r(a_1+\cdots+a_nr^{n-1})=0.
$
This equation yields that $r$ is a unit of $A$. Thus the endmorphism $r=\phi^{-1}\circ i:A\to A$ is an automorphism, $i$ is an isomorphism, and $A=X$.
\end{proof}

Assume that $R$ is complete and has a inifite residue field. Then there is an $R$-submodule $K$ of $Q(R)$ such that $R\subset K \subset \overline{R}$, and as an $R$-module, $K$ is a canonical module of $R$; see \cite[Corollary 2.9]{GMP}. Using this module $K$, we have the following theorem, which essensially proves 
Theorem \ref{1} (3).

\begin{thm} \label{th}
Assume that $R$ is complete and the residue field $R/\mathfrak{m}$ is infinite. Let $K$ be an $R$-submodule of $Q(R)$ such that $R\subset K \subset \overline{R}$, and as an $R$-module, $K$ is a canonical module of $R$. Then the following conditions are equivalent.
\begin{enumerate}[\rm(1)]
\item
$R$ is almost Gorenstein.
\item
$K:\mathfrak{m}=E$.
\item
$\mathfrak{m}:K=\mathfrak{m}$.
\item
$K:E$ is isomorphic to $\mathfrak{m}$.
\item
$K:E$ is in $\ocm'(R)$.
\item
$\ocm'(R)=\cm(E)$.
\end{enumerate}
\end{thm}

\begin{proof}
(1) $\implies$ (2): If $R$ is Gorenstein, then there is nothing to prove. Thus we assume that $R$ is  almost Gorenstein but not Gorenstein. In this case, \cite[Theorem 3.16]{GMP} says that $(E=)\mathfrak{m}:\mathfrak{m}=K:\mathfrak{m}$.

(2) $\implies$ (3): We can easely see $\mathfrak{m}=K:(K:\mathfrak{m})=K:E$. Thus $\mathfrak{m}:K=(K:E):K=(K:K):E=R:E=\mathfrak{m}$.

(3) $\implies$ (1): We have $\mathfrak{m}K\subset \mathfrak{m}\subset R$. This implies that $R$ is almost Gorenstein by \cite[Theorem 3.11]{GMP}.

(2) $\implies$ (4): We have already got $\mathfrak{m}=K:E$. In particular, $K:\mathfrak{m}$ is isomorphic to $\mathfrak{m}$.

(4) $\implies$ (2): We assume that $R$ is not Gorenstein. Then $K:\mathfrak{m}\subset R[K]\subset \overline{R}$ by \cite[Corollary 3.8]{GMP}. On the other hand, $K:\mathfrak{m}$ is isomorphic to $(K:K):E=E$ by the assumption. Applying Lemma \ref{11} to $X=K:\mathfrak{m}$ and $A=E$, we obtain $K:\mathfrak{m}=E$.

(4) $\iff$ (5) and (5) $\iff$ (6): These follow by Lemmas \ref{8} and \ref{5}.
\end{proof}

\begin{proof}[Proof of Theorem \ref{1} {\rm (3)}]
By Corollary \ref{6}, we may assume that $R$ is complete and has infinite residue field. Thus the assersion follows by Theorem \ref{th}.
\end{proof}

We end the section by giving some examples of rings $R$ with $\ocm(R)$ is of finite type.

\begin{ex} \label{ex1}
Let $R=k[[t^3,t^7,t^{11}]]$. Then $E=k[[t^3,t^4]]$, which is Gorenstein. Therefore $\ocm(E)=\cm(E)$ and hence $\ocm(E)=\ocm'(R)=\cm(E)$ by Theorem \ref{1} (1). It follows that $R$ is almost Gorenstein and has minimal multiplicity by Theorem \ref{1} (2) and (3). Note that $\cm(E)$ is of finite type. This implies that $\ocm'(R)$ has finitely many indecomposable objects and thus $\ocm(R)$ is also of finite type.
\end{ex}

\begin{ex} \label{ex2}
Let $R=k[[t^3,t^7,t^8]]$. Then $E=k[[t^3,t^4,t^5]]$ and $\cm(E)$ is of finite type. This yields $\ocm(R)$ is also of finite type. On the other hand, $R$ is not almost Gorenstein.
\end{ex}

\section{The dimension of the stable category of $\cm(R)$}

In this section, as an application of our results obtained in the preceding sections, we consider the dimension of the stable category of Cohen--Macaulay $R$-modules in the case where $R$ is Gorenstein. We shall construct for each integer $n$ an example of a Gorenstein ring $R$ such that $\dim\underline{\cm}(R)=n$. 

Assume that $R$ is a complete domain. Then $E$ is a finite product of local rings because $R$ is henselian and $E$ is module-finite over $R$. Since $E$ is a subring of $\overline{R}$, it is also a domain. Thus $E$ is a complete local domain. Using this argument inductively, we obtain a family $\{(R_n,\mathfrak{m}_n)\}_n$ of local rings with
\begin{equation*}
R_n=\left\{
\begin{array}{ll}
E & \text{if }n=1,\\
\End_{R_{n-1}}(\mathfrak{m}_{n-1}) & \text{otherwise}.
\end{array}
\right.
\end{equation*}

If $R$ is Gorenstein, then as is well known, $\cm(R)$ is a Frobenius exact category. Thus its stable category $\underline{\cm}(R)$ is a triangulated category; see \cite[Chapter 1]{H} for details. In this case, we can consider the dimension of the triangulated category $\underline{\cm}(R)$ in the sense of Rouquier \cite{R}.

\begin{prop} \label{dim}
Assume that $R$ is a complete Gorenstein domain. If $\cm (R_n)$ is of finite type, then $\dim \underline{\cm} (R)\leq n-1$. 
\end{prop}

\begin{proof}
By the assumption, there is an $R_n$-module $G'$ such that $\cm(R_n)=\add (G')$. Let $G=G'\oplus R_{n-1}\oplus\cdots\oplus R_1$.

We claim that $\underline{\cm}(R)=\langle G\rangle_n$. In fact, let $M_1\in \cm(R)$. To show that $M_1\in \langle G \rangle_n$, we may assume that $M_1$ is in $\cm'(R)$. Since $R$ is Gorenstein, $\cm'(R)=\ocm'(R)$ and hence $\cm'(R)=\cm(R_1)$ by Theorem \ref{1} (1). Therefore $M_1\in \cm(R_1)$. Taking a minimal free resolution of $M_1$ over $R_1$, we have an exact sequence $0\to M_2 \to R_1^{\oplus} \to M \to 0$. Here $M_2$ is in $\ocm'(R_1)$ by Lemma \ref{3} and thus in $\cm(R_2)$ by Theorem \ref{1} (1). Applying this argument inductively, we obtain $R$-modules $M_i\in\cm(R_i)$ with exact sequences 
\begin{equation} \label{cl1}
0\to M_i \to R_{i-1}^{\oplus} \to M_{i-1} \to 0
\end{equation}
for each integer $i\geq 2$. Since $M_n\in \cm(R_n)$, it is obvious that $M_n\in \langle G\rangle_1$. Then induction shows that $M_i\in \langle G \rangle_{n-i+1}$ by considering the exact triangle in $\underline{\cm}(R)$ induced by the sequence \eqref{cl1} for each $i$. In particular, $M_1\in \langle G \rangle_n$. 

This claim yields that $\dim\underline{\cm}(R)\leq n-1$.
\end{proof}

\begin{ex}
Let $n\geq 1$ be an integer and $R=k[[H]]$ be the numerical semigroup ring with $H=\langle 2^{n+1},\{2^{n+1}+2^i\}_{i=1,\dots,n},2^{n+1}+3\rangle$. Then $R$ is a complete intersection of codimension $n+1$. Thus $\dim\underline{\cm}(R)\geq n$ by \cite[Corollary 5.10]{BIKO}. On the other hand, we can see that $R_{n+1}$ is isomorphic to $k[[t^2,t^3]]$. Note that $\cm(k[[t^2,t^3]])$ is of finite type. Therefore $\dim \underline{\cm}(R)\leq n$ by Proposition \ref{dim}. Consequently, one has $\dim \underline{\cm}(R)=n$.
\end{ex}

\begin{rem} \label{dim2}
Let $R$ be as in the above example. Since $R_{n+1}=k[[t^2,t^3]]$, $R_{n+2}$ is equal to $k[[t]]$, which is the integral closure of $R$ in its quotient field. Therefore, the endmorphism ring $\Gamma:=\End_R(R\oplus R_1\oplus\cdots R_{n+2})$ has global dimension less than or equal to $n+2$; see \cite[Example 2.2.3(2)]{I}and \cite[Theorem 4]{L}. It is known that $\dim \underline{\cm}(R)+2 \leq \mathop{\mathsf{gl.dim}}\nolimits \Gamma$. This shows that $\dim \underline{\cm}(R) \leq n$.
\end{rem}

\begin{ac}
The author is grateful to his supervisor Ryo Takahashi for giving him helpful advice throughout this paper. The author also thanks to Osamu Iyama for providing helpful suggestions on Theorem \ref{1} and Remark \ref{dim2}.
\end{ac}

%%%%%%%%%%%%%%%%%%%%%%%%%%%%%%%%%%%%%%%%%%%%%%%%%%%%%%%%%%%%%%%%%%%%%%%%%%%%%%%%%%%%%%%%%%%%%%%%%%%%%%%%%%%%%%%%%%


\begin{thebibliography}{1}
\bibitem{ADS}
{\sc M. Auslander; S. Ding; \O. Solberg}, Liftings and weak liftings of modules, {\em J. Algebra} {\bf 156} (1993), 273--317.
\bibitem{B}
{\sc H. Bass}, On the ubiquity of Gorenstein rings, Math. Zeitschrift {\bf 82} (1963), 8--28.
\bibitem{BIKO}
{\sc P. A. Bergh; S. B. Iyengar; H. Krause; S. Oppermann}, Dimensions of triangulated categories
via Koszul objects, {\em Math. Z}. {\bf 265} (2010), no. 4, 849--864.
\bibitem{BV}
{\sc J. P. Brennan; W. V. Vasconcelos}, On the structure of closed ideals, {\em Math. Scand.} {\bf 88} (2001),
no. 1, 3--16.
\bibitem{EG}
{\sc E. G. Evans; P. Griffith}, Syzygies, London Mathematical Society Lecture Note Series, {\bf 106},
{\em Cambridge University Press, Cambridge}, 1985.
\bibitem{ncr}
{\sc H. Dao; O. Iyama; R. Takahashi; C. Vial}, Non-commutative resolutions and Grothendieck groups, {\em J. Noncommut. Geom.} {\bf 9} (2015), no. 1, 21--34.
\bibitem{GMP}
{\sc S. Goto; N. Matsuoka; T. T. Phuong}, Almost Gorenstein rings, {\em J. Algebra}, {\bf 379} (2013), 355--381.
\bibitem{H}
{\sc D. Happel}, Triangulated categories in the representation theory of finite-dimensional algebras, London Mathematical Society Lecture Note Series, 119, {\em Cambridge University Press, Cambridge}, 1988.
\bibitem{HK}
{\sc J. Herzog; E. Kunz}, Der kanonische Modul eines. Cohen-Macaulay-Rings, Lecture Notes in Mathematics {\bf 238}, Springer--Verlag, 1971.
\bibitem{I}
{\sc O.Iyama}, Rejective subcategories of artin algebras and orders, \texttt{arXiv:031128}.
\bibitem{IW}
{\sc O. Iyama; M. Wemyss}, The classification of special Cohen--Macaulay modules, {\em Math. Z.} {\bf 265} (2010), no. 1, 41--83.
\bibitem{K}
{\sc T. Kobayashi}, On delta invariants and indices of ideals, \texttt{arXiv:1705.05042}.
\bibitem{L}
{\sc G. J. Leuschke}, Endomorphism rings of finite global dimension. Canad. {\em J. Math.} {\bf 59} (2007), no. 2, 332-–342.
\bibitem{LW}
{\sc G. J. Leuschke; R. Wiegand}, Cohen--Macaulay Representations, Mathematical Surveys and Monographs, vol. 181, {\em American Mathematical Society, Providence, RI}, 2012.
\bibitem{O}
{\sc A. Ooishi}, On the self-dual maximal Cohen--Macaulay modules, J. Pure Appl. Algebra, {\bf 106}, (1996), 93--102.
\bibitem{R}
{\sc R. Rouquier}, Dimensions of triangulated categories, {\em J. K-Theory} {\bf 1} (2008), no. 2, 193--256.
\end{thebibliography}
\end{document}